\documentclass[a4paper]{article}

\usepackage{amsfonts}
\usepackage{graphicx}
\usepackage{amsthm}
\usepackage{amsmath}
\usepackage{amssymb}
\usepackage{enumerate}

\newtheorem{proposition}{Proposition}[section]
\newtheorem{lemma}[proposition]{Lemma}

\newtheorem{theorem}[proposition]{Theorem}
\newtheorem{definition}[proposition]{Definition}

\begin{document}

\begin{center}
\LARGE
\textbf{The Weisfeiler-Leman algorithm and the diameter of Schreier graphs}
\bigskip\bigskip

\large
Daniele Dona\footnote{The author was partially supported by the European Research Council under Programme H2020-EU.1.1., ERC Grant ID: 648329 (codename GRANT).}
\bigskip

\normalsize
Mathematisches Institut, Georg-August-Universit\"at G\"ottingen

Bunsenstra\ss e 3-5, 37073 G\"ottingen, Germany

\texttt{daniele.dona@mathematik.uni-goettingen.de}
\bigskip\bigskip\bigskip
\end{center}

\begin{minipage}{110mm}
\small
\textbf{Abstract.} We prove that the number of iterations taken by the Weisfeiler-Leman algorithm for configurations coming from Schreier graphs is closely linked to the diameter of the graphs themselves: an upper bound is found for general Schreier graphs, and a lower bound holds for particular cases, such as for Schreier graphs with $G=\mbox{SL}_{n}({\mathbb F}_{q})$ ($q>2$) acting on $k$-tuples of vectors in ${\mathbb F}_{q}^{n}$; moreover, an exact expression is found in the case of Cayley graphs.
\medskip

\medskip

\end{minipage}
\bigskip

\section{Introduction}

The problem of determining the diameter of certain graphs related to finite groups has been extensively investigated in recent years, in the attempt to prove or disprove Babai's conjecture \cite{BS88} on the polylogarithmic bound of the diameter of $\mbox{Cay}(G,S)$ in terms of $|G|$. Recent developments include \cite{HS14} (for $G=\mbox{Sym}(n)$), \cite{BGT11} \cite{PS16} (for $G=SL_{n}({\mathbb F}_{q})$ with bounded $n$) and \cite{BY15} (for $G=SL_{n}({\mathbb F}_{q})$ with bounded $q$): a good part of current research in the area is concentrated towards achieving better bounds in the case of linear algebraic groups with large rank $n$.

Another core problem in graph theory is the question of how to efficiently determine whether two graphs are isomorphic and describe the set of isomorphisms from one to the other; a recent development in the area showed that it is possible to solve this problem in quasipolynomial time in the number of vertices \cite{Ba15} \cite{He17}. Some of the key ideas at the base of the result involve finding more and more refined structures in the target graphs, so that once asymmetry is found (or artificially introduced) the problem can be reduced to smaller sub-problems; one of the first techniques of this sort is the iterative algorithm developed by Weisfeiler and Leman \cite{WL68}, which refines particular colourings of graphs and encodes information about the structure of the graph itself in the new colouring.

The result we present here touches both areas. We work with Schreier graphs of finite groups, giving them a suitable colouring, and we apply the Weisfeiler-Leman algorithm to them: as it turns out, the number of iterations taken before stopping is tightly related to the diameter of the graphs themselves. An upper bound for the number of iterations is found in the case of general Schreier graphs, while a lower bound holds for some interesting particular cases, such as for Schreier graphs with $G=\mbox{SL}_{n}({\mathbb F}_{q})$ ($q>2$) acting on $k$-tuples of vectors in ${\mathbb F}_{q}^{n}$: moreover, this is independent from $n$ and $q$ when $\mbox{char}({\mathbb F}_{q})>2$. In the case of Cayley graphs, instead of bounds we will be able to find an exact expression for the number of iterations as a function of the diameter.

\begin{center}
***
\end{center}

Let $G$ be a finite group and let $S$ be a set of generators of $G$ such that $S=S^{-1}$ and $e\in S$: the \textit{Cayley graph} $\mbox{Cay}(G,S)$ is defined as the graph having $G$ as its set of vertices and $\{(g,sg)|g\in G, s\in S\}$ as its set of edges; since $S$ generates $G$, the graph is strongly connected: using the identity $(g'g^{-1})g=g'$, there exists a directed path from the vertex $g$ to the vertex $g'$ determined by those $s_{i}\in S$ such that $\prod_{i=1}^{m}s_{i}=g'g^{-1}$ (this path is not unique in general, of course).

Cayley graphs are special cases of a more general class of graphs. Again, let $G$ be a finite group and let $S$ be a set of generators of $G$ such that $S=S^{-1}$ and $e\in S$; let $V$ be a set on which $G$ acts transitively: the \textit{Schreier graph} $\mbox{Sch}(V,S)$ is defined as the graph having $V$ as its set of vertices and $\{(v,sv)|v\in V, s\in S\}$ as its set of edges (as it is defined, the Schreier graph is in fact a multigraph);  since $S$ generates $G$ and $G$ acts transitively on $V$, the graph is strongly connected: if $gv=v'$ for some $g\in G$, there exists a directed path from the vertex $v$ to the vertex $v'$ determined by those $s_{i}\in S$ such that $\prod_{i=1}^{m}s_{i}=g$. A Cayley graph is just a Schreier graph where $V=G$ and the action is the usual group multiplication (on the left in our definition, but it does not really matter).

The construction of these graphs and the choice of a symmetric set of generators containing $e$ allow us to see $\mbox{Cay}(G,S)$ and $\mbox{Sch}(V,S)$ as a different type of structure.

\begin{definition}
A (classical) configuration ${\mathfrak X}$ is a pair $(\Gamma, c:\Gamma^{2}\rightarrow {\mathcal C})$ ($\Gamma$ is a finite set of vertices, ${\mathcal C}$ is a finite set of colours) such that:
\begin{enumerate}[{ (i)}]
 \item for any $c\in {\mathcal C}$, if for some $v\in \Gamma$ we have $c(v,v)=c$, then for all $v_{1},v_{2}\in \Gamma$ such that $c(v_{1},v_{2})=c$ we have $v_{1}=v_{2}$;
 \item for any $c\in {\mathcal C}$ there exists a $c^{-1}\in {\mathcal C}$ such that for any $v_{1},v_{2}\in \Gamma$ with $c(v_{1},v_{2})=c$ we have $c(v_{2},v_{1})=c^{-1}$.
\end{enumerate}
\end{definition}

The word ``classical'' in the definition comes from the fact that a more general definition is often used, defining the colouring as $c:\Gamma^{k}\rightarrow {\mathcal C}$ and consequently with some differences in how to define conditions (i) and (ii). It is also to be noted that this is a ``weak'' version of the definition of configuration, as provided in Sun-Wilmes \cite{SW16} and Babai \cite{Ba15}, as opposed to the ``strong'' version that can be found in Helfgott \cite{He17}: in that paper, a stronger assumption on condition (ii) was necessary in order to prove properties of non-classical configurations that here are not needed. By property (i), in a configuration we can distinguish between \textit{vertex colours} (colours coming from $c(v,v)$) and \textit{edge colours} (colours coming from $c(v_{1},v_{2})$ with $v_{1}\neq v_{2}$): these names come from the natural observation that we can think of a configuration as a particular colouring of the complete graph of $|\Gamma|$ vertices, giving to $v$ the colour $c(v,v)$ and to the directed edge $(v_{1},v_{2})$ the colour $c(v_{1},v_{2})$ and noticing that a vertex and an edge will always have different colours in this situation.

There is a natural way to define a configuration ${\mathfrak X}_{C}$ from $\mbox{Cay}(G,S)$: $\Gamma$ can be defined to be the group $G$, while the colouring is given by $c(g_{1},g_{2})=g_{2}g_{1}^{-1}$ if $g_{2}g_{1}^{-1}\in S$ and $c(g_{1},g_{2})=\emptyset$ otherwise; in this case then ${\mathcal C}=S\cup\{\emptyset\}$ (or ${\mathcal C}=S$ in the trivial case $S=G$). ${\mathfrak X}_{C}$ is a configuration (in the weak sense): thanks to $e\in S$ the only vertex colour is $e$ and all the others are edge colours, while thanks to $S=S^{-1}$ the inverse of $s\in S$ as a colour is exactly $s^{-1}$ (and the inverse of $\emptyset$ is $\emptyset$); since we have only one vertex colour ${\mathfrak X}_{C}$ is also a configuration in the strong sense, but it does not make any difference.

In a similar fashion we can define a configuration ${\mathfrak X}_{S}$ from $\mbox{Sch}(V,S)$: $\Gamma$ can be defined to be the set $V$, while the colouring is given by $c(v_{1},v_{2})=\{s\in S|sv_{1}=v_{2}\}$; in this case then ${\mathcal C}\subseteq {\mathcal P}(S)$. ${\mathfrak X}_{S}$ is a classical configuration (in the weak sense, but not in the strong sense): to prove that it satisfies (i), notice that if for a colour $c$ we have $c(v,v)=c$ then $e\in c$, so that for any other two vertices $v_{1},v_{2}$ with $c(v_{1},v_{2})=c$ we have $v_{1}=v_{2}$ (in other words, vertex colours are exactly those who contain $e$); to prove that it satisfies (ii), observe that for all $c\in {\mathcal C}$ we have a natural definition $c^{-1}=\{s^{-1}|s\in c\}$, thanks to $S=S^{-1}$. If we see Cayley graphs as particular Schreier graphs, the configurations ${\mathfrak X}_{C}$ and ${\mathfrak X}_{S}$ built on the same $\mbox{Cay}(G,S)$ are clearly isomorphic, with each colour $s\neq\emptyset$ in ${\mathfrak X}_{C}$ corresponding to $\{s\}$ in ${\mathfrak X}_{S}$.

As mentioned before, we now introduce a more refined type of structure.

\begin{definition}
A (classical) coherent configuration is a pair ${\mathfrak X}=(\Gamma, c:\Gamma^{2}\rightarrow{\mathcal C})$ that satisfies (i) and (ii) and such that:
\begin{itemize}
 \item[\ (iii)] for every $c_{0},c_{1},c_{2}\in {\mathcal C}$ there is a constant $\gamma=\gamma(c_{0},c_{1},c_{2})\in {\mathbb N}$ such that for every $v_{1},v_{2}\in \Gamma$ with $c(v_{1},v_{2})=c_{0}$ the number of $w\in \Gamma$ with $c(v_{1},w)=c_{1}$ and $c(w,v_{2})=c_{2}$ is $\gamma$ (independently from the choice of $v_{1},v_{2}$).
\end{itemize}
\end{definition}

The colouring of a coherent configuration contains much more information about its structure than the one coming from a usual configuration. Especially important to us is the following result:

\begin{proposition}[\cite{He17}, Exercise 2.16(a)]\label{Prwalkcol}
Let ${\mathfrak X}=(\Gamma, c:\Gamma^{2}\rightarrow{\mathcal C})$ be a coherent configuration, and let $c_{0},c_{1},...,c_{k}$ be a sequence of colours with $k\geq 2$. Then there is a constant $\gamma=\gamma(c_{0},c_{1},...,c_{k})\in {\mathbb N}$ such that for every $v_{1},v_{2}\in \Gamma$ with $c(v_{1},v_{2})=c_{0}$ the number of $k$-tuples $(w_{1},...,w_{k-1})\in \Gamma^{k}$ with $c(v_{1},w_{1})=c_{1}$, $c(w_{i-1},w_{i})=c_{i}$ for all $1<i<k$ and $c(w_{k-1},v_{2})=c_{k}$ is $\gamma$ (independently from the choice of $v_{1},v_{2}$).
\end{proposition}

So an edge colour $c(v_{1},v_{2})$ in a coherent configuration not only knows by definition about colourings of triangles $v_{1},w,v_{2}$, but knows also about colourings of walks $v_{1},w_{1},...,w_{k-1},v_{2}$ of any length.

\begin{proof}
We proceed by induction on $k$: for $k=2$ the statement is exactly condition (iii) of coherent configurations, so there is nothing to prove.

Suppose now that this is true for $k$. We are given $v_{1},v_{2}$ with $c(v_{1},v_{2})=c_{0}$ and we have to find the number of walks of colours $c_{1},c_{2},...,c_{k+1}$ from $v_{1}$ to $v_{2}$: such a walk however is merely the composition of two walks $c_{1},c_{2},...,c_{k-1}$ and $c_{k},c_{k+1}$, so we can just consider any walk $c_{1},c_{2},...,c_{k-1},c'$ of length $k$ from $v_{1}$ to $v_{2}$ (for all $c'$) and for each of them any triangle of colours $c',c_{k},c_{k+1}$ built on $(w_{k-2},v_{2})$ of colour $c'$; the composition of these two structures will give us the desired walk of length $k+1$. The constants $\gamma$ for walks of length $2$ and $k$ are independent from the choice of initial vertices, thus the same will occur for $k+1$:

\begin{equation*}
\gamma(c_{0},c_{1},...,c_{k},c_{k+1})=\sum_{c'\in {\mathcal C}}\gamma(c_{0},c_{1},...,c_{k-1},c')\gamma(c',c_{k},c_{k+1})
\end{equation*}
\end{proof}

A configuration, and in particular the configurations ${\mathfrak X}_{C},{\mathfrak X}_{S}$ that we are going to study, is not necessarily coherent. There is a natural way to refine a configuration into a coherent configuration, through the algorithm devised by Weisfeiler and Leman in \cite{WL68}.

We first define ${\mathcal C}^{(0)}={\mathcal C}$ as the colour set at the $0$-th iteration; we can define ${\mathcal C}^{(h+1)}$ from ${\mathcal C}^{(h)}$ by calling $c^{(h+1)}(v_{1},v_{2})$ the tuple: 

\begin{equation}\label{eqwl}
\left(c^{(h)}(v_{1},v_{2}), \left(\left|\left\{w\in V|c^{(h)}(v_{1},w)=c_{1},c^{(h)}(w,v_{2})=c_{2}\right\}\right|\right)_{c_{1},c_{2}\in {\mathcal C^{(h)}}}\right)
\end{equation}

The colouring $c^{(h+1)}$ is more refined than $c^{(h)}$; notice that $c^{(h+1)}(v_{1},v_{2})$ contains as information, for each choice of $c_{1},c_{2}\in {\mathcal C}^{(h)}$, the number of vertices $w$ as in condition (iii) (which is not yet independent from the choice of $v_{1},v_{2}$). Once we reach an iteration where there is no refinement, it means that all these numbers are the same for each pair $(v_{1},v_{2})$ with the same colour, i.e. the configuration has become coherent: the Weisfeiler-Leman algorithm stops at this point.

One last observation is necessary with regard to Weisfeiler-Leman.

\begin{proposition}\label{Prisowl}
Define ${\mathfrak X}^{(h)}$ as the configuration at the $h$-th step of Weisfeiler-Leman. Then $\mbox{\normalfont Aut}({\mathfrak X}^{(h)})=\mbox{\normalfont Aut}({\mathfrak X}^{(h+1)})$.
\end{proposition}

\begin{proof}
As $c^{(h+1)}$ is a refinement of $c^{(h)}$, we have already $\mbox{Aut}({\mathfrak X}^{(h)})\supseteq \mbox{Aut}({\mathfrak X}^{(h+1)})$. On the other side, if $\sigma\in \mbox{Aut}({\mathfrak X}^{(h)})$, then for any pair $v_{1},v_{2}\in V$ and any pair $c_{1},c_{2}\in {\mathcal C}^{(h)}$ each vertex $w$ with $(c(v_{1},w),c(w,v_{2}))=(c_{1},c_{2})$ is sent to a vertex $\sigma(w)$ such that $(c(\sigma(v_{1}),\sigma(w)),c(\sigma(w),\sigma(v_{2})))=(c_{1},c_{2})$: this implies that the numbers in (\ref{eqwl}) are also preserved by $\sigma$, therefore $\sigma\in \mbox{Aut}({\mathfrak X}^{(h+1)})$ too.
\end{proof}

This explains why Weisfeiler-Leman is so interesting in the context of the Graph Isomorphism Problem: after using the algorithm we have a more refined colouring, which means that we have more possibility to exploit the subtle differences between two graphs, but at the same time the algorithm is designed to preserve all isomorphisms, so that when we prove that a certain $\sigma$ is not an isomorphism for the final coherent configurations then we have proved that it is also not an isomorphism for the original graphs.

We state now our main results.

\begin{theorem}\label{Thupper}
Let $G$ be a finite group and let $S$ be a set of generators of $G$ with $e\in S=S^{-1}$. Suppose that $G$ acts transitively on a set $V$, and consider the configuration ${\mathfrak X}_{S}$ coming from $\mbox{\normalfont Sch}(V,S)$; then the number $WL({\mathfrak X}_{S})$ of nontrivial iterations of the Weisfeiler-Leman algorithm satisfies:
\begin{equation*}
WL({\mathfrak X}_{S}) \leq \log_{2}\mbox{{\normalfont diam Sch}}(V,S)+3
\end{equation*}
\end{theorem}

where by counting nontrivial iterations we merely want to ignore the last one with no colour refinement.

Together with this upper bound, lower bounds also hold in some more limited but still very interesting cases. The scope of the lower bound is explicitly stated later (see Theorem~\ref{Thlowerg}), but here we specialize it to a situation that is relevant in the context of Babai's conjecture.

\begin{theorem}\label{ThlowerS}
Let $G=SL_{n}({\mathbb F}_{q})$ with $q>2$ and let $S$ be a set of generators of $G$ with $I_{n}\in S=S^{-1}$; for any $0<k<n$, let $V$ be the set of linearly independent $k$-tuples of vectors of ${\mathbb F}_{q}^{n}$, with the action of $G$ on $V$ defined as $A(v_{1},...,v_{k})=(Av_{1},...,Av_{k})$. Consider the configuration ${\mathfrak X}_{S}$ coming from $\mbox{\normalfont Sch}(V,S)$; then, if $p$ is the smallest prime such that $p|(q-1)$, the number $WL({\mathfrak X}_{S})$ of nontrivial iterations of the Weisfeiler-Leman algorithm satisfies:
\begin{equation*}
WL({\mathfrak X}_{S}) \geq \log_{2}\mbox{{\normalfont diam Sch}}(V,S)-\log_{2}(p-1)-3
\end{equation*}
\end{theorem}

Notice that this result does not depend on $n$, and we have dependence on $q$ only when $\mbox{char}({\mathbb F}_{q})=2$ (since otherwise $p=2$ and $\log_{2}(p-1)=0$).

Finally, we state a result that gives an exact expression for the number of iterations for any Cayley graph.

\begin{theorem}\label{ThCay}
Let $G$ be a finite group and let $S$ be a set of generators of $G$ with $e\in S=S^{-1}$. Consider the configuration ${\mathfrak X}_{C}$ coming from $\mbox{\normalfont Cay}(G,S)$; then the number $WL({\mathfrak X}_{C})$ of nontrivial iterations of the Weisfeiler-Leman algorithm satisfies:
\begin{equation*}
WL({\mathfrak X}_{C}) = \begin{cases}
 \lceil \log_{2}(\mbox{{\normalfont diam Cay}}(G,S)-1) \rceil & \mbox{ \ if \ }\forall g\exists! g'\ \mbox{with } d(g,g')\mbox{ the diameter} \\
 \lceil \log_{2}\mbox{{\normalfont diam Cay}}(G,S) \rceil & \mbox{ \ otherwise}
\end{cases}
\end{equation*}
\end{theorem}

We remark that both Theorem~\ref{Thupper} and a bound analogous to the one in Theorem~\ref{ThlowerS} would hold for general Cayley graphs, although as expected they are weaker than the theorem above.

\section{The upper bound}
We first prove the upper bound in Theorem~\ref{Thupper}; in fact, we prove the same upper bound for a more general class of configurations, of which the configurations ${\mathfrak X}_{C},{\mathfrak X}_{S}$ that we defined from $\mbox{Cay}(G,S)$ and $\mbox{Sch}(V,S)$ are just particular cases.

\begin{theorem}\label{Thupperg}
Let ${\mathfrak X}$ be a configuration; call an edge colour $c\in {\mathcal C}$ nonempty if for every $v\in \Gamma$ there is at most one $w\in \Gamma$ with $c(v,w)=c$, and call it empty otherwise. Suppose that the coloured graph $\Gamma_{{\mathfrak X}}=(\Gamma,\{(v_{1},v_{2})\in \Gamma^{2}|c(v_{1},v_{2}) \mbox{ nonempty}\})$ is connected. Then the number $WL({\mathfrak X})$ of nontrivial iterations of the Weisfeiler-Leman algorithm satisfies:
\begin{equation*}
WL({\mathfrak X})\leq \log_{2}\mbox{{\normalfont diam }}\Gamma_{{\mathfrak X}}+3
\end{equation*}
\end{theorem}

\begin{proof}[Proof that Thm~\ref{Thupperg} $\Rightarrow$ Thm~\ref{Thupper}]
Consider the configuration ${\mathfrak X}_{S}$ coming from the graph $\mbox{Sch}(V,S)$: the only possible empty colour is $\emptyset$ (hence the name) and all the other edge colours are nonempty, because for any $v\in V$ and $s\in S$ there is evidently only one $v'$ with $sv=v'$, so by our definition for any colour $c\subseteq S$, $c\neq\emptyset$ at most one $v'$ would realize $c(v,v')=c$ for any fixed $v$. If $\emptyset$ is indeed empty, by construction a nonempty-coloured pair in $\Gamma_{{\mathfrak X}_{S}}$ corresponds to an edge (or multiedge) in $\mbox{Sch}(V,S)$: thus connectedness of $\mbox{Sch}(V,S)$ implies the same for $\Gamma_{{\mathfrak X}_{S}}$, and the two have the same diameter. If $\emptyset$ is nonempty, i.e. in the extreme case when $S$ sends any $v$ to any $v'$ except at most one, we have $\mbox{diam }\Gamma_{{\mathfrak X}_{S}}=1$ and $\mbox{diam }\mbox{Sch}(V,S)=2$, so the bound still holds; if $\emptyset$ is not a colour at all, $\mbox{diam }\Gamma_{{\mathfrak X}_{S}}=\mbox{diam }\mbox{Sch}(V,S)=1$.
\end{proof}

From now on, for the sake of clarity, colours that appear during the iterations of the Weisfeiler-Leman algorithm as refinements of empty (resp. nonempty) colours are still called empty (resp. nonempty), even if some new empty colours could now satisfy the nonemptyness criterion: as a matter of fact, it is a key point in the success of the argument that eventually \textit{all} new colours will be nonempty according to our definition of the word; however our need to refer ourselves to the origin of the intermediate colours is more pressing than highlighting the acquisition of the property of nonemptyness. Observe that, by the construction of the new colours in (\ref{eqwl}), each of them knows every past colour from which it descended, so that recognizing whether a given intermediate colour is empty or nonempty does not create any problem.

In order to prove Theorem~\ref{Thupperg}, we need the next lemma. Inside ${\mathfrak X}$, call \textit{walk} (of length $l$) any sequence of consecutive pairs of vertices $(w_{0},w_{1}),(w_{1},w_{2}),\ldots,$ $(w_{l-1},w_{l})$ (with their respective colours) which corresponds to a walk of length $l$ in $\Gamma_{{\mathfrak X}}$; equivalently, a walk in ${\mathfrak X}$ is any sequence of consecutive pairs with only nonempty colours.

\begin{lemma}\label{Leknowwalk}
For every $v_{1},v_{2}\in \Gamma$, at the $k$-th iteration, $c^{(k)}(v_{1},v_{2})$ knows all walks of length $\leq 2^{k}$ from $v_{1}$ to $v_{2}$.
\end{lemma}

\begin{proof}
We proceed by induction on $k$. When $k=0$ the statement is trivial, since the only walk of length $1$ from $v_{1}$ to $v_{2}$ that could possibly exist is the edge $(v_{1},v_{2})$ provided that its colour is nonempty, which the colour $c^{(0)}(v_{1},v_{2})$ evidently knows.

Suppose that the statement has been proved for $k$, and consider any walk of length $\leq 2^{k+1}$ from $v_{1}$ to $v_{2}$: for any such walk, there exists a $w$ that splits the original walk into two walks (from $v_{1}$ to $w$ and from $w$ to $v_{2}$) of length $\leq 2^{k}$; the existence of these walks is an information contained inside $c^{(k)}(v_{1},w)$ and $c^{(k)}(w,v_{2})$ respectively, so at the $(k+1)$-th iteration $c^{(k+1)}(v_{1},v_{2})$ will know about the existence of this pair of walks (and consequently of the original long walk).
\end{proof}

Compare the statement of Lemma~\ref{Leknowwalk} with that of Proposition~\ref{Prwalkcol}: according to the latter, colours of a coherent configuration know all walks of any length, while the former describes how, iteration after iteration, the colours of a configuration arrive to gain knowledge of longer and longer walks, whose length doubles at every step. The origin of the $\log_{2}$ of the diameter in all our results clearly resides in this ``learning'' process; proving the lower and the upper bounds involves making sure that such process is, in a sense, more or less respectively necessary and sufficient to make the configuration ${\mathfrak X}$ coherent.

As for the upper bound, it turns out that, given a vertex $v$, knowing which walks starting from $v$ reach the same endpoint is enough information (in a graph like $\Gamma_{{\mathfrak X}}$) to reconstruct a piece of the graph around $v$.

\begin{lemma}\label{Lereconst}
Fix any vertex $v\in \Gamma_{{\mathfrak X}}$ and suppose that we know, for any two sequences of nonempty colours of length $\leq k$, whether the walks starting from $v$ defined by these sequences exist and have the same endpoint. Then it is possible to reconstruct in a unique way the subgraph $\Gamma^{(k)}(v)\subseteq\Gamma_{{\mathfrak X}}$ given by the edges at distance $\leq k$ from $v$; in other words, for any other coloured graph $\Gamma'$ whose walks of length $\leq k$ from a certain vertex $w$ satisfy the same conditions of existence and equality of endpoints, there exists a unique graph isomorphism $\Gamma^{(k)}(v)\rightarrow\Gamma'^{(k)}(w)$ sending $v$ to $w$.
\end{lemma}

\begin{proof}
The information provided to us is basically just a collection of strings of nonempty colours $c_{1}c_{2}...c_{l\leq k}$, meaning that  ``the walk starting from $v$ consisting of consecutive edges of colour $c_{1},c_{2},...,c_{l\leq k}$ exists'', and a collection of pairs of strings of nonempty colours $(c_{1}c_{2}...c_{l\leq k},c'_{1}c'_{2}...c'_{l'\leq k})$, meaning that ``the walk starting from $v$ consisting of consecutive edges of colour $c_{1},c_{2},...,c_{l\leq k}$ has the same endpoint as the walk starting from $v$ consisting of consecutive edges of colour $c'_{1},c'_{2},...,c'_{l'\leq k}$''. With these strings and pairs of strings in our hands, we will manage to rebuild what the graph $\Gamma_{{\mathfrak X}}$ looks like up to distance $k$.

We proceed by induction on $k$. For $k=1$, walks of length $\leq 1$ are just single edges: there is no possible equality of two endpoints, so $\Gamma^{(1)}(v)$ is just a star whose internal vertex is $v$ and whose leaves are all the edges $(v,v_{i})$ of $\Gamma_{{\mathfrak X}}$; there will obviously be then a unique isomorphism to any other star with internal vertex $w$ and leaves with the same colour as $\Gamma^{(1)}(v)$.

Suppose now that the statement is true for $k$, and consider for $k+1$ the two graphs $\Gamma^{(k+1)}(v),\Gamma'$: we already have a partial isomorphism from $\Gamma^{(k)}(v)\subseteq\Gamma^{(k+1)}(v)$ to the subgraph given by the edges at distance at most $k$ from $w$ in $\Gamma'$, and we just need to extend this isomorphism to the edges and vertices at distance $k+1$. Every relation among walks whose edges are already covered by walks of length $\leq k$ clearly does not endanger the already existing isomorphism, so we need to consider only relations involving a walk that includes an edge at distance $k+1$ from $v$ or $w$ (which will necessarily be the last edge).

Each edge at distance $k+1$ in $\Gamma^{(k+1)}(v)$ is $(v_{1},v_{2})$ where either $d(v,v_{1})=d(v,v_{2})=k$ or $d(v,v_{1})=k$ and $d(v,v_{2})=k+1$. In the first case, the two vertices $v_{1},v_{2}$ are already sent in a unique way to $w_{1},w_{2}\in \Gamma'$; there exist walks of length $k$ from $v$ to $v_{2}$, so for any string $c_{1}c_{2}...c_{k}c_{k+1}$ whose last edge would be $(v_{1},v_{2})$ in $\Gamma_{{\mathfrak X}}$ there are pairs of the type $(c_{1}c_{2}...c_{k}c_{k+1},c'_{1}c'_{2}...c'_{k})$: seeing now inside $\Gamma'$ the walks $c_{1}c_{2}...c_{k},c'_{1}c'_{2}...c'_{k}$ of length $k$, whose endpoints are already known to be $w_{1}$ and $w_{2}$, the pairs described above mean that ``the edge starting from $w_{1}$ of colour $c_{k+1}$ ends in $w_{2}$''. This implies that it is in fact possible to extend the isomorphism to the edge $(v_{1},v_{2})$ previously not covered by it, since this edge has the same colour as the edge $(w_{1},w_{2})$; in this way the original isomorphism is now extended to all edges of this type, and all relations among walks containing these edges now trivially agree with this isomorphism.

In the second case, the only relations that involve these type of edges are pairs of the form $(c_{1}c_{2}...c_{k}c_{k+1},c'_{1}c'_{2}...c'_{k}c'_{k+1})$; evidently it is possible to group these strings of length $k+1$ into equivalence classes corresponding to their endpoints $v_{2}$ inside $\Gamma^{(k+1)}(v) \setminus \Gamma^{(k)}(v)$. Consider now all strings $c_{1}^{i}c_{2}^{i}...c_{k}^{i}c_{k+1}^{i}$ inside a given class: their subwalks of length $k$ end in vertices $v_{1}^{i}\in \Gamma_{{\mathfrak X}}$, which uniquely correspond to vertices $w_{1}^{i}\in \Gamma'$; therefore the relations given by all pairs of $c_{1}^{i}c_{2}^{i}...c_{k}^{i}c_{k+1}^{i}$, when seen in $\Gamma'$, would signify that ``the edges starting from $w_{1}^{i}$ of colour $c_{k+1}^{i}$ all end in the same vertex'' (say $w_{2}$): we extend then the isomorphism by sending each $v_{2}$ to the corresponding $w_{2}$ and each $(v_{1}^{i},v_{2})$ to the corresponding $(w_{1}^{i},w_{2})$, thus covering all vertices at distance $k+1$ and all edges of this second type. The isomorphism has been extended in a unique way to the whole $\Gamma^{(k+1)}(v)$, so we are done.
\end{proof}

Notice how important it is that we are working with nonempty colours, i.e. in a situation where every vertex has at most one adjacent edge of any given colour: only with this condition we can talk about \textit{the} edge starting from $v$ and having colour $c$, or about \textit{the} walk starting from $v$ and having colours $c_{1},c_{2},...,c_{l}$; without it, there would be no guarantee that there is an isomorphism as the one we constructed, i.e. we would not be able to deduce the shape of the subgraph $\Gamma^{(k)}(v)$ only by looking at the information about the walks, because more than one graph could satisfy the same conditions: this is crucial in the proof of the upper bound.

We are now ready to prove Theorem~\ref{Thupperg}.

\begin{proof}[Proof of Thm~\ref{Thupperg}]
Define $k=\lceil\log_{2}(\mbox{diam }\Gamma_{{\mathfrak X}}+1)\rceil$: for any $v_{1},v_{2}\in \Gamma_{{\mathfrak X}}$, $c^{(k)}(v_{1},v_{2})$ knows all walks of length $\leq 2^{k}$ from $v_{1}$ to $v_{2}$ by Lemma~\ref{Leknowwalk}; since $d(v_{1},v_{2})\leq \mbox{diam }\Gamma_{{\mathfrak X}}<2^{k}$, there exists at least one of these walks: this implies that we have $c^{(k)}(v_{1},v_{2})\neq c^{(k)}(v_{1},v'_{2})$ for any $v_{2}\neq v'_{2}$, because the walk from a given vertex defined by a given sequence of nonempty colours is unique (if it exists).

At the next iteration, the colour $c^{(k+1)}(v_{1},v_{1})$ knows the number of vertices $w$ such that $(c^{(k)}(v_{1},w),c^{(k)}(w,v_{1}))=(c_{1},c_{2})$ for any choice of colours $c_{1},c_{2}\in {\mathcal C}^{(k)}$, so in particular it knows whether there is a colour $c_{1}$ containing one or two given walks (with the expression ``the colour $c_{1}$ contains a given walk'' we mean that two vertices $v,v'$ with $c(v,v')=c_{1}$ would have this walk going from $v$ to $v'$) such that there is one $w$ with $c^{(k)}(v_{1},w)=c_{1}$: in other words, the colour $c^{(k+1)}(v_{1},v_{1})$ knows if a walk starting from $v_{1}$ of length $\leq 2^{k}$ given by a certain sequence of nonempty colours exists, thanks to the information about colours containing one given walk, and it knows also if two walks starting from $v_{1}$ of length $\leq 2^{k}$ defined by given sequences of nonempty colours have the same endpoint, thanks to the information about colours containing two given walks; so we are in the situation described in the statement of Lemma~\ref{Lereconst}, where at the $(k+1)$-th iteration from the colour $c^{(k+1)}(v_{1},v_{1})$ we can reconstruct the subgraph $\Gamma_{{\mathfrak X}}^{(2^{k})}(v_{1})$.

After one more iteration, also each pair of distinct vertices $(v_{1},v_{2})$ will have a colour capable of reconstructing the subgraph $\Gamma_{{\mathfrak X}}^{(2^{k})}(v_{1})$, because $c^{(k+2)}(v_{1},v_{2})$ will in turn contain $c^{(k+1)}(v_{1},v_{1})$ in a way that allows us to identify it unequivocally: remember, in a configuration vertex colours and edge colours are distinct, so $c^{(k+1)}(v_{1},v_{1})$ can be defined as the unique vertex colour $c$ such that the number of $w$ with $\left(c^{(k+1)}(v_{1},w),c^{(k+1)}(w,v_{2})\right)=\left(c,c^{(k+1)}(v_{1},v_{2})\right)$ is nonzero.

Now, $\Gamma_{{\mathfrak X}}^{(2^{k})}(v_{1})$ is actually the whole $\Gamma_{{\mathfrak X}}$ for any $v_{1}$, since every edge is at distance at most $\mbox{diam }\Gamma_{{\mathfrak X}}+1$ from $v_{1}$: this means that at the $(k+2)$-th iteration we can look at any colour of any vertex or edge and reconstruct the whole graph $\Gamma_{{\mathfrak X}}$ around this vertex or edge. But then, if two vertices or two edges have still the same colour after $k+2$ iterations, there is an automorphism of $\Gamma_{{\mathfrak X}}$ sending one into the other.

We still have to check that this extends to an automorphism of ${\mathfrak X}$. We remark that this passage is not necessary to prove the special cases ${\mathfrak X}_{C},{\mathfrak X}_{S}$ that we are interested in, since in that case there is at most one empty colour $\emptyset$ and any automorphism preserving all nonempty colours (i.e. an automorphism of $\Gamma_{{\mathfrak X}}$) would preserve all colours (i.e. it would be an automorphism of ${\mathfrak X}$).

For any pair $(v_{1},v_{2})$ and any $c(v,w)$ empty, there is a sequence of consecutive pairs that goes first through a walk from $v_{1}$ to $v$, then includes the pair $(v,w)$, then through another walk goes from $w$ to $v_{2}$; if $c^{(k+2)}(v_{1},v_{2})=c^{(k+2)}(v'_{1},v'_{2})$, the resulting automorphism of $\Gamma_{{\mathfrak X}}$ sends the walk $v_{1}\rightarrow v$ to a walk $v'_{1}\rightarrow v'$ (and analogously for the second walk), and all these walks can be chosen to be at most of length $\mbox{diam }\Gamma_{{\mathfrak X}}<2^{k}$. The colour $c^{(k)}(v_{1},v)$ knows the walk $v_{1}\rightarrow v$, therefore $c^{(k+1)}(v_{1},w)$ knows the colour of the whole sequence composed by $v_{1}\rightarrow v$ and the pair $(v,w)$; on the other side, $c^{(k)}(w,v_{2})$ knows the walk $w\rightarrow v_{2}$, so that $c^{(k+2)}(v_{1},v_{2})$ knows the whole long sequence $v_{1}\rightarrow v\rightarrow w\rightarrow v_{2}$: $c^{(k+2)}(v_{1},v_{2})=c^{(k+2)}(v'_{1},v'_{2})$ then implies $c(v,w)=c(v',w')$ (as always, we have repeatedly used the uniqueness of nonempty-coloured walks from a given vertex: otherwise we would have not been able only from colours to identify $(v',w')$ as the image of $(v,w)$). We have extended the automorphism of $\Gamma_{{\mathfrak X}}$ to the empty colours, i.e. it is in fact an automorphism of the whole configuration ${\mathfrak X}$.

By Proposition~\ref{Prisowl}, the existence of this automorphism means that if two pairs $(v_{1},v_{2}),(v'_{1},v'_{2})$ have the same colour at the $(k+2)$-th iteration they will always have the same colour, i.e.:

\begin{equation*}
WL({\mathfrak X})\leq k+2\leq \log_{2}\mbox{{\normalfont diam }}\Gamma_{{\mathfrak X}}+3
\end{equation*}

\end{proof}

\section{The lower bound}

We now prove the lower bound in Theorem~\ref{ThlowerS}; again, we prove the same for a more general class of configurations than the ones given by Cayley and Schreier graphs.

\begin{theorem}\label{Thlowerg}
Let ${\mathfrak X}$ be a configuration with only one empty colour, and let $\Gamma_{{\mathfrak X}}$ be defined as in Theorem~\ref{Thupperg}; suppose that there exists a $\varphi\in\mbox{\normalfont Aut}({\mathfrak X}),\varphi\neq\mbox{\normalfont Id}_{{\mathfrak X}}$ with the property that each nontrivial $\varphi^{i}\in\langle\varphi\rangle$ has no fixed points, and consider such a $\varphi$ with minimal $|\langle\varphi\rangle|$. Then the number $WL({\mathfrak X})$ of nontrivial iterations of the Weisfeiler-Leman algorithm satisfies:
\begin{equation*}
WL({\mathfrak X})\geq \log_{2}\mbox{{\normalfont diam }}\Gamma_{{\mathfrak X}}-\log_{2}(|\langle\varphi\rangle|-1)-3
\end{equation*}
\end{theorem}

\begin{proof}[Proof that Thm~\ref{Thlowerg} $\Rightarrow$ Thm~\ref{ThlowerS}]
As we already said in the previous section, the only possible empty colour is $\emptyset$; the extreme cases of $\emptyset$ nonempty and $\emptyset$ not a colour give a trivially true bound in Theorem~\ref{ThlowerS} since $\mbox{diam Sch}(G,S)=2$ and $1$ respectively, while when $\emptyset$ is empty we have $\mbox{diam }\Gamma_{{\mathfrak X}_{S}}=\mbox{diam Sch}(G,S)$. As for the automorphism condition, if $q>2$ the elements $h\in {\mathbb F}_{q}\setminus\{0,1\}$ induce automorphisms $\varphi_{h}$ defined by $h(v_{1},...,v_{k})=(hv_{1},...,hv_{k})$: they are really automorphisms since $Av=v' \Leftrightarrow Ahv=hv'$ and they obviously do not fix any point since $hv\neq v$ for $h\neq 1$.  The multiplicative group ${\mathbb F}_{q}^{*}$ is just a cycle of order $q-1$, so there will be an element of order $p$ as defined in the statement of Theorem~\ref{ThlowerS}, and $p$ will provide an upper bound for $\mbox{min}_{\varphi}\{|\langle\varphi\rangle|\}$.
\end{proof}

Again, notice that Theorem~\ref{Thlowerg} would also apply to Cayley graphs: the same reasoning applies with regard to the colour $\emptyset$, and right multiplication by any $h\in G$ gives an automorphism $\varphi_{h}$ that behaves as required.

The key idea to prove Theorem~\ref{Thlowerg} is the following: for any $\varphi$ with the property described in the statement, the various pairs $(v,\varphi^{i}(w))$ despite being distinct from one another will all have the same colour until we manage to ``cover the distance'' between $v$ and at least one of the $\varphi^{i}(w)$, i.e. to encode inside the colours information about walks of length $d(v,\varphi^{i}(w))$. This is what we meant when we said that the process of bestowing upon our colours knowledge of the walks was necessary in order to reach coherence through Weisfeiler-Leman: we prove that, since the surroundings of the $\varphi^{i}(w)$ are indistinguishable, from the point of view of $v$ we will not be able to differentiate among them until we manage to touch at least one of them.

This idea translates to the following result.

\begin{lemma}\label{Lew-wequal}
For any pair of vertices $v,w\in \Gamma_{{\mathfrak X}}$ and for any integer $k\geq 0$ such that $d(v,\varphi^{i}(w))>2^{k}$ for all $i$, we have that $c^{(k)}(v,\varphi^{i}(w))$ is the same for all $i$.
\end{lemma}

\begin{proof}
We proceed by induction on $k$. For $k=0$ the statement is obvious, since by hypothesis $d(v,\varphi^{i}(w))>1$ for all $i$ and $c^{(0)}$ is the original colouring of ${\mathfrak X}$, so that $c^{(0)}(v,\varphi^{i}(w))=\emptyset$ for each of these pairs, where $\emptyset$ is the unique empty colour.

Now suppose that the statement is true for $k$, i.e. for every two elements $v,w\in \Gamma_{{\mathfrak X}}$ such that $\forall i \left(d(v,\varphi^{i}(w))>2^{k}\right)$ the colour $c^{(k)}(v,\varphi^{i}(w))$ is the same for all $i$; we fix now elements $v,w$ that satisfy the condition $\forall i \left(d(v,\varphi^{i}(w))>2^{k+1}\right)$ and we prove for them the statement for $k+1$.

The colour $c^{(k+1)}(v,w)$ consists of the previous colour $c^{(k)}(v,w)$ and of the number of $v'$ such that $\left(c^{(k)}(v,v'),c^{(k)}(v',w)\right)=(c_{1},c_{2})$ for each choice of $c_{1},c_{2}\in {\mathcal C}^{(k)}$: since by hypothesis we already have the same $c^{(k)}(v,\varphi^{i}(w))$ for all $i$, we only need to prove that for each pair $(c_{1},c_{2})$ the numbers are the same for all $(v,\varphi^{i}(w))$, i.e. for each $i$ there exists a bijection $\sigma_{i}:V\rightarrow V$ such that:

\begin{equation}\label{eqbije}
\left(c^{(k)}(v,v'),c^{(k)}(v',w)\right)=\left(c^{(k)}(v,\sigma_{i}(v')),c^{(k)}(\sigma_{i}(v'),\varphi^{i}(w))\right)
\end{equation}

Define $A^{(k)}=\left\{v'\in V|\min_{i}d(\varphi^{i}(v),v')\leq 2^{k}\right\}$: we show that the bijections then can be defined to be:

\begin{equation*}
\sigma_{i}(v')=\begin{cases}
 v' & \mbox{ \ if } v'\in A^{(k)} \\
 \varphi^{i}(v') & \mbox{ \ if } v'\not\in A^{(k)}
\end{cases}
\end{equation*}

These are really bijections, since the fact that $\varphi\in\mbox{Aut}({\mathfrak X})$ preserves distances implies $v'\not\in A^{(k)}\Leftrightarrow \varphi^{i}(v')\not\in A^{(k)}$. Figure~1 shows what happens in our situation: since $w$ is at least $2^{k+1}$ far from all the $\varphi^{i}(v)$ and the walks that the colours know have length at most $2^{k}$, either we are close to one of the $\varphi^{i}(v)$ and then far away from $w$ (which allows us to ``confuse'' $w$ and $\varphi^{i}(w)$ from the point of view of $v'$) or we are far away from the $\varphi^{i}(v)$ (and we can confuse $v'$ and $\varphi^{i}(v')$ from the point of view of $v$).

\begin{figure}
\includegraphics[width=\textwidth]{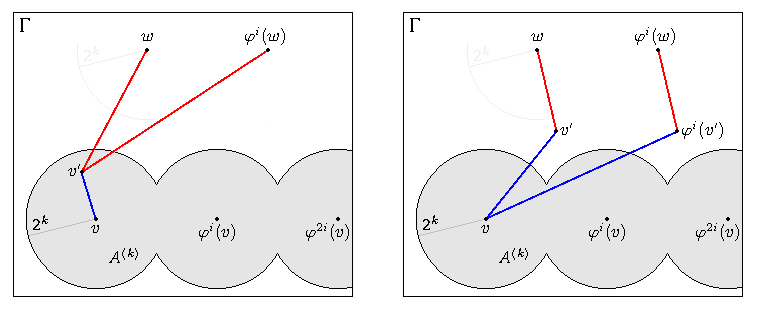}
\caption{Colours $c^{(k)}$ for the two cases  $v'\in A^{(k)}$ and  $v'\not\in A^{(k)}$.}
\end{figure}

When $v'\in A^{(k)}$, the first components in (\ref{eqbije}) are already identical. By definition, there is a $j$ such that $d(\varphi^{j}(v),v')\leq 2^{k}$; then using $d(\varphi^{j}(v),\varphi^{i}(w))=d(v,\varphi^{i-j}(w))$ we have:

\begin{equation*}
d(v',\varphi^{i}(w))\geq d(\varphi^{j}(v),\varphi^{i}(w))-d(\varphi^{j}(v),v')>2^{k+1}-2^{k}=2^{k}
\end{equation*}
thus obtaining equality for the second components in (\ref{eqbije}) by inductive hypothesis.

When $v'\not\in A^{(k)}$, we have $d(v,\varphi^{i}(v'))=d(\varphi^{-i}(v),v')>2^{k}$ for all $i$ by definition: this gives us already $c^{(k)}(v,v')=c^{(k)}(v,\varphi^{i}(v'))$ by inductive hypothesis, and the first components in (\ref{eqbije}) are equal; equality of the second components is also clear, because $\varphi^{i}$ is an automorphism of ${\mathfrak X}$ that sends $(v',w)$ to $(\varphi^{i}(v'),\varphi^{i}(w))$, which means that their colours will always be equal to each other at every iteration (as consequence of Proposition~\ref{Prisowl}).

We have thus shown that the $\sigma_{i}$ are bijections satisfying (\ref{eqbije}): this means that the colours $c^{(k+1)}(v,\varphi^{i}(w))$ are all the same for every $i$, proving the inductive step.
\end{proof}

It is now easy to prove Theorem~\ref{Thlowerg} from this lemma.

\begin{proof}[Proof of Thm~\ref{Thlowerg}]
Suppose that we have reached the end of the Weisfeiler-Leman algorithm, i.e. we have reached an iteration $k$ that does not refine any colour in ${\mathcal C}^{(k-1)}$. For any two elements $v,w_{1}\in \Gamma_{{\mathfrak X}}$ there is a walk of nonempty colours from $v$ to $w_{1}$ that is unique to $w_{1}$ (by definition of nonempty colour) in the sense that no other $w_{2}$ will have the same walk going from $v$ to $w_{2}$; this means in particular that the final colours $c^{(k)}(v,\varphi^{i}(w))$ must be all distinct for fixed $v,w$, now that we have reached the end of Weisfeiler-Leman and the configuration has become coherent: in fact for each of them there is a walk that is not shared by the others, so by Proposition~\ref{Prwalkcol} these walks will give different colours to the pairs $(v,\varphi^{i}(w))$. By Lemma~\ref{Lew-wequal} however these pairs have all the same colour when $\forall i\left(d(v,\varphi^{i}(w))>2^{k}\right)$, which implies that for any fixed $v\in \Gamma_{{\mathfrak X}}$ and for any other $w\in \Gamma_{{\mathfrak X}}$ there exists an $i$ such that $\varphi^{i}(w)$ has distance at most $2^{k}$ from $v$.

We now suppose that the cyclic group generated by $\varphi$ in $\mbox{Aut}({\mathfrak X})$ is of minimal size, the only property of $\varphi$ required in the statement that we have not used up to this point. Fix now a $v$ such that $\mbox{diam }\Gamma_{{\mathfrak X}}$ is realized as a distance by some pair $(v,v')$, and fix a $w$ such that $d(v,w)=2^{k}+1$ (of course we are supposing that the diameter is larger than $2^{k}$, otherwise the bound of the theorem is already true): we must have $d(v,\varphi^{i}(w))\leq 2^{k}$ for some $i$, so changing the name of the generator of this shortest cycle $\langle\varphi\rangle$ we can suppose that $d(v,\varphi(w))\leq 2^{k}$. For a $v'$ such that $d(v,v')=\mbox{diam }\Gamma_{{\mathfrak X}}$ we must also have $d(v,\varphi^{j}(v'))\leq 2^{k}$ for some $j$; therefore:

\begin{eqnarray*}
\mbox{diam }\Gamma_{{\mathfrak X}} & = & d(v,v') \ \ = \ \ d(\varphi^{j}(v'),\varphi^{j}(v)) \ \ \leq \\
 & \leq & d(\varphi^{j}(v'),v)+\sum_{i=1}^{j}\left(d(\varphi^{i-1}(v),\varphi^{i}(w))+d(\varphi^{i}(w),\varphi^{i}(v))\right) \ \ =  \\
 & = & d(v,\varphi^{j}(v'))+j(d(v,\varphi(w))+d(v,w)) \ \ \leq \\
 & \leq & 2^{k}+j(2^{k}+2^{k}+1) \ \ \leq \ \ 2^{k+2}j \ \ \leq \ \ 2^{k+2}(|\langle\varphi\rangle|-1)
\end{eqnarray*}

and we have the bound on the last nontrivial iteration $k-1$:

\begin{equation*}
 WL({\mathfrak X}) \ = \ k-1 \ \geq \ \log_{2}\mbox{diam }\Gamma_{{\mathfrak X}}-\log_{2}(|\langle\varphi\rangle|-1)-3
\end{equation*}

\end{proof}

\section{The case of Cayley graphs}
Now we prove Theorem~\ref{ThCay}. The case of a Cayley graph $\mbox{{\normalfont Cay}}(G,S)$ is particularly nice with respect to general Schreier graphs: the reason behind this fact is ultimately the existence of automorphisms (namely, right multiplications by elements of $G$) that can send any $g$ to any $g'$, as we will be able to observe during the proof of the result.

We start with a lemma that provides an upper bound for the number of iterations.

\begin{lemma}\label{LeupperC}
The right hand side of the equality in Theorem~\ref{ThCay} provides a sufficient number of iterations to make the configuration ${\mathfrak X}_{C}$ coherent (i.e. the inequality with $\leq$ in Theorem~\ref{ThCay} holds).
\end{lemma}

\begin{proof}
Right multiplication by any $h\in G$ gives an automorphism $\varphi_{h}$ of ${\mathfrak X}_{C}$, since $sg=g' \Leftrightarrow sgh=g'h$ for any $g,g'\in G$, $s\in S$; by Proposition~\ref{Prisowl}, at the end of the algorithm $\varphi_{h}$ will still be an automorphism of the final coherent configuration, which means in particular that every vertex will have the same colour and that for any three vertices $g,g',h\in G$ there exists an $h'\in G$ such that $c(g,h)=c(g',h')$. Therefore, the maximum possible colour refinement that we can expect to obtain from Weisfeiler-Leman is the one where all sets $\left\{c(g,h)|h\in G\right\}$ are equal for every $g$ but where any two colours $c(g,h),c(g,h')$ are distinct for $h\neq h'$.

Indeed, this is the colouring that we reach at the end: as we have already observed in the general Schreier case, any coloured walk from $g$ to $h$ (corresponding to a certain product of generators that represents $g^{-1}h$) is the unique walk from $g$ consisting of that sequence of colours, and by Proposition~\ref{Prwalkcol} the final colour $c(g,h)$ knows it. Thus, every $c(g,h)$ knows a walk that all other $c(g,h')$ do not know, and the colouring described above is achieved.

When every colour $c(g,h)$ for all $h\in G$ has become distinct from the others, we have undoubtedly reached the end of the algorithm; this happens when the colour of every pair $(g,h)$ knows at least one walk connecting them, or at the very least when for a fixed $g$ all but one of them know such a walk (so that the remaining pair $(g,h)$ has the unique colour that knows no walk: this corresponds, informally speaking, to the situation where even the emptiest descendant of the colour $\emptyset$ is nonempty according to the definition given in Theorem~\ref{Thupperg}). By Lemma~\ref{Leknowwalk}, this happens when at the $k$-th iteration we have $2^{k}$ at least as large as the diameter, or at least $\mbox{{\normalfont diam Cay}}(G,S)-1$ if for any $g$ there is only one $h$ whose distance from $g$ is the diameter; the result follows.
\end{proof}

To prove an inequality in the other direction, we make use of the abundance of automorphisms in the Cayley graph to prove a stronger version of Lemma~\ref{Lew-wequal}.

\begin{lemma}\label{Lehh'equal}
For any four vertices $g,h,g',h'\in G$ and for any integer $k\geq 0$ such that $d(g,h),d(g',h')>2^{k}$, we have $c^{(k)}(g,h)=c^{(k)}(g',h')$.
\end{lemma}

\begin{proof}
Again, we proceed by induction on $k$. For $k=0$ the statement is obvious, because all pairs $(g,h)$ of vertices with distance $>1$ have the same colour $\emptyset$ at the $0$-th step.

Now suppose that the statement is true for $k$. First, we are going to prove that for any three vertices $g,h,h'\in G$ with $d(g,h),d(g,h')>2^{k+1}$ the two pairs $(g,h),(g,h')$ will still have the same colour at the $(k+1)$-th step; the idea is, as in Lemma~\ref{Lew-wequal}, to construct a suitable bijection $\sigma:G\rightarrow G$ with a property analogous to (\ref{eqbije}), namely:

\begin{equation}\label{eqbijeC}
\left(c^{(k)}(g,g'),c^{(k)}(g',h)\right)=\left(c^{(k)}(g,\sigma(g')),c^{(k)}(\sigma(g'),h')\right)
\end{equation}

Define:

\begin{equation*}
\sigma(g')=\begin{cases}
 g' & \mbox{ \ if } d(g',h),d(g',h')>2^{k} \\
 g'h^{-1}h' & \mbox{ \ if } d(g',h)\leq 2^{k} \\
 \tau(g') & \mbox{ \ if } d(g',h)>2^{k},d(g',h')\leq 2^{k}
\end{cases}
\end{equation*}

where $\tau$ is an arbitrary bijection from the set $\{g'\in G|d(g',h)>2^{k},d(g',h')\leq 2^{k}\}$ to the set $\{g'\in G|d(g',h)\leq 2^{k},d(g',h')>2^{k}\}$.

In the first case we have obviously a bijection, whose image is the set of all vertices of distance $>2^{k}$ from both $h$ and $h'$; also, by inductive hypothesis in this set we have $c^{(k)}(g',h)=c^{(k)}(g',h')$, so (\ref{eqbijeC}) is satisfied. In the second case, right multiplication by $h^{-1}h'$ is an automorphism (hence a bijection) from the ball around $h$ to the one around $h'$ both of radius $2^{k}$: this descends trivially from the fact that if we have $s_{i}\in S$ such that $\left(\prod_{i}s_{i}\right)g'=h$ then also $\left(\prod_{i}s_{i}\right)g'h^{-1}h'=h'$, so that in particular $d(g',h)=d(g'h^{-1}h',h')$; moreover for the same reason we must have $c^{(k)}(g',h)=c^{(k)}(g'h^{-1}h',h')$, because for every walk given by a sequence of colours $s_{i}$ from $g'$ to $h$ the same walk exists from $g'h^{-1}h'$ to $h'$: as observed in the proof of Lemma~\ref{LeupperC}, for our three vertices $g',h,g'h^{-1}h'$ there must be a fourth $x$ that will have $c(g',h)=c(g'h^{-1}h',x)$ at the end, and $x=h'$ is the only possible candidate. We also have $c^{(k)}(g,g')=c^{(k)}(g,g'h^{-1}h')$ by inductive hypothesis, since $g'$ and $g'h^{-1}h'$ are at distance $\leq 2^{k}$ from $h$ and $h'$ (both at distance $>2^{k+1}$ from $g$); thus, (\ref{eqbijeC}) is satisfied again. In the third case, $\tau$ is a bijection because the balls around $h'$ and $h$ of radius $2^{k}$ have the same number of vertices, and domain and codomain of $\tau$ are these two balls minus their intersection; in addition, the colours of the four pairs $(g,g'),(g',h),(g,\tau(g')),(\tau(g'),h')$ are all the same since each of their distances is $>2^{k}$, so we have (\ref{eqbijeC}) for this case too. Given that the codomains in the three cases are disjoint, $\sigma$ is indeed a bijection satisfying (\ref{eqbijeC}), which implies that $c^{(k+1)}(g,h)=c^{(k+1)}(g,h')$.

Now we know that for any vertex $g$ there is a colour $c_{g}$ such that all vertices $h$ at distance $>2^{k+1}$ will have $c^{(k+1)}(g,h)=c_{g}$: but then it is obvious that $c_{g}$ does not depend on $g$, since any $g$ is sent to any $g'$ by some automorphism that will preserve distances in the graph, so that $(g,h)$ of distance $>2^{k+1}$ is sent to some $(g',h')$ of same distance (and consequently $c_{g}=c_{g'}$). This proves that $c^{(k+1)}(g,h)=c^{(k+1)}(g',h')$ whenever $d(g,h),d(g',h')>2^{k+1}$, concludes the inductive step and proves the lemma.
\end{proof}

Now we can easily prove Theorem~\ref{ThCay}.

\begin{proof}[Proof of Thm~\ref{ThCay}]
We have already shown the $\leq$ direction in Lemma~\ref{LeupperC}. On the other hand, proving the $\geq$ direction means proving that if at the $k$-th iteration there are three vertices $g,h,h'$ such that $d(g,h),d(g,h')>2^{k}$ then there are more iterations to come; by Lemma~\ref{Lehh'equal}, however, in this situation $c^{(k)}(g,h)=c^{(k)}(g,h')$ and we know that at the end of the algorithm we will have $c(g,h)\neq c(g,h')$ (ultimately coming from Proposition~\ref{Prwalkcol}), so the statement holds.
\end{proof}

\section{Concluding remarks}

The Weisfeiler-Leman algorithm behaves especially well on Cayley graphs for the following reason: as pointed out by L. Bartholdi (personal communication), since at the $k$-th iteration the vertices $g'$ at distance $\leq 2^{k}$ from $g$ have pairs $(g,g')$ all of distinct colours while the $g'$ at distance $>2^{k}$ have pairs $(g,g')$ all of the same colour, the configuration obtained after the $k$-th step starting from $\mbox{Cay}(G,S)$ is the same as the configuration coming from $\mbox{Cay}(G,S^{2^{k}})$ itself; thus it is natural to expect that the algorithm would stop when $S^{2^{k}}$ becomes the whole $G$ (or possibly $G$ except for just one $g$).

Schreier graphs, which in general have not as many automorphisms as Cayley graphs, satisfy less stringent upper and lower bounds. In one direction, having vertices with different stabilizers can make pairs of vertices of large distance receive different colours early on, so that the colouring at the $k$-th step of the algorithm conveys more information than the mere colouring coming from the choice of $S^{2^{k}}$ as generators. Consider for example the set $V={\mathbb Z}/n{\mathbb Z}$ and the group $G=\mbox{Sym}(n)$ generated by the set $S$ consisting of all transpositions of the type $(i \ i+1)$ and the identity: the diameter of $\mbox{Sch}(V,S)$ is $\lfloor \frac{n}{2} \rfloor$, but since every $i\in V$ is stabilized by a different subset of the generators (i.e. all of them except $(i-1 \ i)$ and $(i \ i+1)$) after the first step all the pairs of colour $\emptyset$ are differentiated immediately; therefore the number of iterations in this case will be $1$, for any choice of $n$ (the bound given by Theorem~\ref{Thlowerg} also fails, because there are no non-trivial automorphisms of the coloured graph).

On the other hand, the fact that the maximum possible colour refinement that we can expect from Weisfeiler-Leman is more than the one described during the proof of Lemma~\ref{LeupperC} could mean that more steps are necessary than just the ones needed to reach the end of the graph: the information to reconstruct the whole graph (to put it in the language of Lemma~\ref{Lereconst}) exists already but it could be scattered among the various pairs of vertices of the graph and it could take a few more steps to make sure that every single pair knows everything about the graph. Consider for example the set $V=\{1,2,\ldots,14\}$ and the set $S=\{e,\sigma^{\pm 1},\tau^{\pm 1}\}$ (the group $G$ has little importance here: for the sake of simplicity, think of it as the free group $F_{2}$, or as a suitable subgroup of $\mbox{Sym}(14)$) acting on $V$ as follows:

\begin{eqnarray*}
\sigma: & & 1 \mapsto 2 \mapsto 3 \mapsto 4 \mapsto 5 \mapsto 6 \mapsto 7 \mapsto 1, \\
 & & 8 \mapsto 9 \mapsto 10 \mapsto 11 \mapsto 12 \mapsto 13 \mapsto 14 \mapsto 8 \\
\tau: & & 1 \mapsto 8 \mapsto 9 \mapsto 2 \mapsto 1, \ 3 \mapsto 10 \mapsto 3, \ 4 \mapsto 11 \mapsto 4, \\
 & & 5 \mapsto 12 \mapsto 5, \ 6 \mapsto 13 \mapsto 6, \ 7 \mapsto 14 \mapsto 7
\end{eqnarray*}

In this situation, $\mbox{Sch}(V,S)$ looks like two coloured heptagons whose corresponding vertices are linked, so that its diameter is $4$; from the reasoning in Lemma~\ref{Leknowwalk} and Lemma~\ref{Lereconst}, after the second iteration of the algorithm there is enough information to reconstruct the whole graph, and this would be in accord with a hypothetical estimate as in Theorem~\ref{ThCay}. Nevertheless, it is possible to verify that we have $c^{(2)}(5,5)=c^{(2)}(12,12)$ and $c^{(3)}(5,5)\neq c^{(3)}(12,12)$, so that the number of iterations for this configuration is $>2$ (it is $3$ indeed).

\begin{center}
***
\end{center}

Theorems~\ref{Thupper}-\ref{ThlowerS}-\ref{ThCay} establish a rather strong correlation between the number of iterations of Weisfeiler-Leman and the diameter of Cayley and Schreier graphs. In particular, Theorem~\ref{ThCay} allows us to describe the diameter of Cayley graphs as a function of $WL({\mathfrak X_{C}})$; it is natural, in the context of Babai's conjecture, to ask ourselves whether it is possible that the number of Weisfeiler-Leman iterations could be reflected in another way in the construction of the graph: usually, determining the runtime of the algorithm would involve from the beginning the actual construction of the graph, thus making it useless for the solution of the conjecture. In light of this, it would be interesting to find results that express $WL({\mathfrak X_{C}})$ as a more intrinsic feature of the construction of Cayley graphs.

On the other side, to the best of our knowledge the results established here are the first ones that determine nontrivial bounds for the number of iterations of the Weisfeiler-Leman algorithm on configurations, either general or of a specific form (the trivial bound on a generic classical configuration being $|\Gamma|^{2}-|{\mathcal C}|$). In this direction, it would be interesting to find results in the style of Theorems~\ref{Thupperg}-\ref{Thlowerg} with different initial conditions: a case that appears to be particularly appealing is the case of non-coloured graphs, for which one can wonder whether it could be possible to bound $WL({\mathfrak X})$ from above by some function of the diameter of the graph, as we have done here for the particular coloured graphs described in the statement of Theorem~\ref{Thupperg}.

\section*{Acknowledgements}

The author thanks H. A. Helfgott for introducing him to both of the research areas touched by the results of this paper.


\begin{thebibliography}{111}
\bibitem{Ba15}
 BABAI L\'aszl\'o, \textit{Graph Isomorphism in Quasipolynomial Time}, available as \texttt{arxiv.org:1512.03547}.
\bibitem{BS88}
 BABAI L\'aszl\'o, SERESS \'Akos, \textit{On the diameter of Cayley graphs of the symmetric group}, J. Combin. Theory Ser. A 49(1) (1988), pp. 175-179.
\bibitem{BY15}
 BISWAS Arindam, YANG Yilong, \textit{A diameter bound for finite simple groups of large rank},  J. London Math. Soc. 95(2) (2017), pp. 455-474.
\bibitem{BGT11}
 BREUILLARD Emmanuel, GREEN Ben, TAO Terence, \textit{Approximate subgroups of linear groups}, Geom. Funct. Anal. 21(4) (2011), pp. 774-819.
\bibitem{He17}
 HELFGOTT Harald Andr\'es, \textit{Isomorphismes de graphes en temps quasi-polynomial} (in French), available as \texttt{arxiv.org:1701.04372}.
\bibitem{HS14}
 HELFGOTT Harald A., SERESS \'Akos, \textit{On the diameter of permutation groups}, Ann. of Math. 179(2) (2014), pp. 611-658.
\bibitem{PS16}
 PYBER L\'aszl\'o, SZAB\'O Endre, \textit{Growth in finite simple groups of Lie type}, J. Amer. Math. Soc. 29(1) (2016), pp. 95-146.
\bibitem{SW16}
 SUN Xiaorui, WILMES John, \textit{Faster canonical forms for primitive coherent configurations}, Proc. 47th STOC (2015), pp. 693-702.
\bibitem{WL68}
 WEISFEILER Boris, LEMAN Andrei, \textit{A reduction of a graph to a canonical form and an algebra arising during this reduction} (in Russian), Nauchno-Technicheskaya Informatsiya 9 (1968), pp. 12-16.
\end{thebibliography}
\end{document}